\documentclass[reqno]{amsart}

\usepackage[utf8]{inputenc}
\usepackage[T1]{fontenc}

\usepackage{amsmath}
\usepackage{amssymb}
\usepackage{amsfonts}
\usepackage{graphicx}
\usepackage{amsthm}
\usepackage{enumerate}
\usepackage{lscape}
\usepackage{dsfont}
\usepackage{color}
\usepackage{mathtools}

\usepackage{setspace}
\newcommand{\Z}{\mathds{Z}}
\newcommand{\Q}{\mathds{Q}}              
\newcommand{\R}{\mathds{R}}

\newcommand{\CP}{\mathds{C}\mathrm{P}}
\newcommand{\N}{\mathds{N}}

\newcommand{\CH}{\mathds{C}\mathrm{H}}

\newcommand{\Ric}{\mathrm{Ric}}

\newcommand{\C}{\mathds{C}}            
\newcommand{\de}{\partial}          

\newcommand{\K}{K\"{a}hler}

\newcommand{\F}{F_\epsilon}

\def\b{\beta}

\def\b1{{\rm id}}

\newtheorem{theor}{Theorem}[section]
\newtheorem{prop}[theor]{Proposition}
\newtheorem{defin}{Definition}
\newtheorem{lem}[theor]{Lemma}
\newtheorem{cor}[theor]{Corollary}

\newtheorem{example}{Example}
\newtheorem{rmk}{Remark}

\begin{document}

\title[Extremal K\"{a}hler metrics induced by  complex space forms]{Extremal K\"{a}hler metrics induced by
finite or infinite dimensional complex space forms}

\author{Andrea Loi}
\address{(Andrea Loi) Dipartimento di Matematica \\
         Universit\`a di Cagliari, Via Ospedale 72, 09124  (Italy)}
         \email{loi@unica.it}

\author{Filippo Salis}
\address{(Filippo Salis) Istituto Nazionale di Alta Matematica, Politecnico di Torino,
Corso Duca degli Abruzzi 24, 10129 Torino (Italy)}
\email{filippo.salis@gmail.com}

\author{Fabio Zuddas}
\address{(Fabio Zuddas) Dipartimento di Matematica \\
         Universit\`a di Cagliari, Via Ospedale 72, 09124  (Italy)}
         \email{ fabio.zuddas@unica.it}

\thanks{
The first and the third authors were supported  by Prin 2015 -- Real and Complex Manifolds; Geometry, Topology and Harmonic Analysis -- Italy,  by STAGE - Funded by Fondazione di Sardegna and by KASBA- Funded by Regione Autonoma della Sardegna. The second author was a research fellow of INdAM.\\
All the three authors were supported by INdAM GNSAGA - Gruppo Nazionale per le Strutture Algebriche, Geometriche e le loro Applicazioni.
}

\subjclass[2000]{53C55, 32Q15, 32T15.} 
\keywords{\K\ \ metric, extremal metric; constant scalar curvature metric;  Calabi's diastasis function; complex space forms}

\begin{abstract}
In this paper we address the problem of studying those complex manifolds $M$ equipped with   extremal metrics $g$ induced by finite or infinite dimensional complex space forms. 
We prove that when $g$ is  assumed to be  radial and the ambient space is  finite dimensional 
then $(M, g)$ is itself a complex space form. 
We  extend this result to the  infinite dimensional setting  by imposing  the strongest assumption that  the metric $g$ 
has  constant scalar curvature and  is  {\em well-behaved} (see Definition \ref{WB} in the Introduction).
 Finally, we analyze the radial   \K-Einstein metrics   induced by infinite dimensional elliptic   complex space forms and we show that if such a  metric is assumed to satisfy a stability condition then it is forced to have constant non-positive holomorphic sectional curvature.  
\end{abstract}
 
\maketitle

\tableofcontents  

\section{Introduction}
Extremal \K\ 
metrics  were introduced  by Calabi \cite{calextrem} in
the compact case as the solution for the variational problem in  a
\K\ class defined by the square integral of the scalar curvature.
Therefore they are a generalization of constant scalar curvature (cscK) and hence of  \K-Einstein (KE) 
metrics. Calabi himself constructs nontrivial extremal (namely with
nonconstant scalar curvature) metrics on some compact manifolds.
In the last thirty years extremal \K\ metrics   were rediscovered  by several  mathematicians due to
their link with the  stability of complex vector bundles (see e.g.
 \cite{burns}, \cite{chtian},    \cite{levin}, \cite{ma} and also the 
 introductory book \cite{SZGA}). 
 The reader is also referred to the recent papers  \cite{blow1}, \cite{arpac}, \cite{blow2} and 
 \cite{blow3} 
for the existence of extremal metrics via blowup constructions.
 
Obviously extremal metrics  cannot be defined in the noncompact
case as  the solutions of a variational problem involving some
integral on the manifold. Nevertheless
they can be  alternatively defined  as those  metrics 
such that the (1,0)-part of the
Hamiltonian vector field associated to the scalar curvature  is
holomorphic.
In the
noncompact case, the existence and uniqueness of  such    metrics
are far from being  understood. For example, in
\cite{cheng} (see also \cite{cheng2}), there has been shown the
existence of a nontrivial extremal   and complete  \K\ metric
 in a complex one-dimensional manifold.
More recently M. Abreu \cite{abreuC1} inspired by the work of Calabi \cite{calextrem} considered cohomogeneity one examples 
of extremal metrics on noncompact manifolds.

In this paper we address the issue of classifying  those (finite dimensional) complex manifolds $M$ admitting an   extremal metric $g$ induced
by  a finite
or infinite dimensional complex space form $(S^{N}, g^{N}_c)$ of constant holomorphic sectional curvature $c$ and complex dimension $N\leq\infty$.
By the word \lq\lq induced'' we mean that the \K\ manifold  $(M, g)$  can be  \K\  immersed  into $(S^{N}, g^{N}_c)$, i.e. there exists a holomorphic map $\varphi :M\rightarrow S^{N}$ such that 
$\varphi^*g^{N}_c=g$ (see \cite{calabi} or the book \cite{LoiZedda-book} for an update material on the subject).

If one assumes that   $(S^{N}, g^{N}_c)$ is  complete and simply-connected one has the corresponding three cases, depending on the sign of $c$:

- for $c=0$, $S^N=\C^N$ ($S^\infty=\ell^2(\C)$) and  $g^N_0$ is the flat metric with   associated \K\ form
\begin{equation}\label{omega0} 
\omega_{0}=\frac{i}{2}\partial\bar\partial |z|^2,\  |z|^2=\sum_{j=1}^N|z_j|^2,\ N\leq\infty;
\end{equation}

- for $c<0$, $S^N=\C H^N$ is  the $N$-dimensional complex hyperbolic  space, namely the unit ball of $\C^N$ with the metric $g_c^N$ with associated  \K\ form 
 \begin{equation}\label{omegahyp} 
\omega_c=\frac{i}{2c}\partial\bar\partial\log (1-|z|^2);
\end{equation}

- for $c>0$, $S^N=\C P^N$ is  the $N$-dimensional complex projective space 
 and $g^N_c$  is the  metric with associated  \K\ form  $\omega_c$, 
 given in homogeneous coordinates by:
 \begin{equation}\label{omegaproj} 
\omega_c=\frac{i}{2c}\partial\bar\partial\log (|Z_0|^2 +\cdots +|Z_N|^2).
\end{equation}
 Notice that when $c=1$ (resp. $c=-1$) the metric 
$g_c^N$ is the standard Fubini-Study metric $g_{FS}$ (respectively hyperbolic metric $g_{hyp}$) of 
holomorphic sectional curvature $4$ (resp. $-4$).
Throughout the paper we will say  that a metric $g$ on a complex (connected) manifold is {\em projectively induced} if $(M, g)$ admits a \K\ immersion
into $(\CP^{N}, g_{FS})$. We say $g$ is  finitely (resp. infinitely) projectively induced if $N<\infty$ (resp. $N=\infty$).

We 
believe that the extremal metrics
induced by a {\em finite} dimensional complex form  are forced to have constant holomorphic sectional curvature
as expressed by the following:

\vskip 0.3cm

\noindent
{\bf Conjecture 1:}
{\em Let $g$ be an extremal metric on an $n$-dimensional complex manifold $M$ induced 
by a finite dimensional complex space form  of constant holomorphic sectional curvature $c$.
Then the following facts hold true:

\noindent
(i) if $c\leq 0$ then  $(M, g)$ is a complex space form of holomorphic sectional curvature $c$ and the   immersion is totally geodesic.

\noindent
(ii) if $c>0$, then $M$  an open subset of a flag manifold\footnote{A flag manifold $(F, g)$ is a compact simply-connected \K\ manifold acted upon transitively by its holomorphic isometries group.} 
$(F, h)$ and $g=h_{|_{M}}$}

\vskip 0.3cm

A possible way to attack (i)  of  Conjecture 1  could be through the following  steps: 
extremal $\rightarrow$ cscK, cscK $\rightarrow$ KE and finally to appeal to   a fundamental result of  M. Umehara \cite{UmearaE}
asserting  that a \K\ immersion of a  KE manifold  into  a finite dimensional  complex space form 
of non positive holomorphic sectional curvature is totally geodesic.
Unfortunately at the moment we are unable to prove any of the two implications.
For part (ii) of Conjecture 1 one 
should try to  show the following  three facts:
extremal $\rightarrow$ cscK, cscK $\rightarrow$ KE and 
KE $\rightarrow$ $h$.
Regarding the step cscK $\rightarrow$ KE
a partial result   for projectively induced metrics 
was obtained by S. Kobayashi \cite{kob} (see also the work of  S.S. Chern \cite{ch} for the case of codimension one immersions)
which shows that  when $M$ is  
a complete intersection in the complex projective space $\C P^N$ and  the restriction of the Fubini--Study metric to $M$ is  cscK then it is  KE (and hence $M$  is either the quadric of $\C P^N$ 
or it is  totally geodesic  by a fundamental result of J. Hano \cite{HANO}).

The proof of the  step  KE $\rightarrow$ $h$  represents an important breakthrough in the classification of finite projectively induced KE metrics.
The only known facts in this direction  are the extension of the above mentioned Chern's result to the codimension $2$ case due to  K. Tsukada \cite{ts}
and the proof of the  positivity of the Einstein constant of a compact KE submanifold of the complex projective space due to D. Hulin \cite{HU} (see also  \cite{SALIS} for the case of rotation invariant metrics  in codimension $3$).

 \subsection{Statements of the main results}
 In this paper we verify  Conjecture 1   under the additional  assumption that the metrics  involved  are 
 {\em  radial} \K\  metrics, i.e. they admit a global K\"ahler potential  $\Phi:M\rightarrow \R$ which depends only on the sum $|z|^2 = |z_1|^2 + \cdots + |z_n|^2$ of the local coordinates' moduli. 
Since $M$ is assumed to be connected this means that there exists a smooth function $f: (r_{\inf}, r_{\sup})\rightarrow \R$, 
 $0\leq r_{\inf}<r_{\sup}\leq\infty$, such that 
 $\Phi (z)=f(r)$ and 
 \begin{equation}\label{omegar}
 \omega =\frac{i}{2} \partial \bar \partial f(r), \ r=|z|^2.
 \end{equation}
 
 The prototype of radial \K\ metrics in  complex dimension $n$  are   the flat metric $g_0$ on $\C^n$, the hyperbolic metric $g_{hyp}$ on $\CH^n$ and the  Fubini-Study metric
 $g_{FS}$ on the affine chart $U_0=\{Z_0\neq 0\}$ with complex coordinates $z_j=\frac{Z_j}{Z_0}$, $j=1, \dots n$.
 
Our first result is then   the following:

\begin{theor}\label{mainteor}
Let $g$ be a radial extremal metric on  a  $n$-dimensional complex manifold $M$.
Assume that $(M, g)$ can be \K\ immersed into a finite dimensional  complex space form 
$(S^N, g^N_c)$. Then 

\noindent
(1) If $c\leq 0$ then  $(M, g)$ is a complex space form of holomorphic sectional curvature $c$ and the   immersion is totally geodesic.

\noindent
(2) If $c>0$ then $M$ is an open subset of $\C P^n$,  $g$ is an integer multiple of $g^n_c$, 
i.e. $g=mg_c^n$, $m\in \N^+$.
\end{theor}

\begin{rmk}\rm
The conclusion  (2) of the  theorem can also  be accompained by an explicit description   of the \K\  immersion given by a suitable normalization of the Veronese embedding (see \cite{calabi} for details).
Notice also that (2) is a particular case of (ii) of Conjecture 1 since it is not hard to see that   a  homogeneous \K\  metric $h$ on a flag manifold  $F$
which admits a radial potential (on an open subset of $F$)
can exist only when $F=\C P^n$ and $g$ is a multiple of $g_{FS}$. 
\end{rmk}

It is worth pointing out  that Theorem \ref{mainteor} is of local nature, i.e. there  are no  topological assumptions on the manifold $M$ and the \K\ immersions 
are not required to be injective.

Since an extremal cohomogeneity one toric \K\ metric $g$ on a  compact complex  manifold  $T$ admits a radial \K\ potential on a dense open subset, we get:
\begin{cor}
If $(T, g)$ is finitely projectively induced
then $(M, g)=(\C P^n, g_{FS})$. 
\end{cor}

Unfortunately without any further assumptions 
Theorem \ref{mainteor} does not   extend 
to  the infinite dimensional setting.
Indeed there exist (see Example \ref{exfond}  in Subsection \ref{Examples} below )  extremal (not cscK) radial metrics  which can be \K\ immersed into 
any infinite dimensional complex space form.

Even when dealing with the strongest assumption  of cscK metrics 
one can exhibit examples of  infinitely projectively induced cscK (not KE) metrics (see Example \ref{simanca} and Example \ref{partbal}  in Subsection \ref{Examples}).

By analyzing these last  examples one discovers two facts:  a) such a metric  cannot be \K\ immersed into any infinite dimensional complex space form of non positive
holomorphic sectional curvature and  b) they do not satisfy the following  definition of fundamental importance for our analysis.

\begin{defin}\label{WB}\rm
A radial   \K\ metric $g$ with radial potential $f:(r_{\inf}, r_{\sup})\rightarrow \R$
 is said to be  {\em well-behaved}\footnote{
Clearly  if a radial metric $g$ is defined at  $r_{\inf}=0$  then it is well-behaved and
in particular, $g_0$, $g_{hyp}$ and $g_{FS}$ (the latter on the affine charts $U_0=\{Z_0\neq 0\}$)  are well-behaved and $r_{\inf}=0$.} 
if $rf'(r)\rightarrow 0$ for $r\rightarrow r_{\inf}^+$. 
\end{defin}

In the following theorem which represents our second result we show that fact a) is true for any cscK metric which is not of constant holomorphic sectional 
curvature  and that well-behaveness is indeed the right condition  to  impose in order for (2) of  Theorem \ref{mainteor} to extend to the infinite dimensional setting.

\begin{theor}\label{mainteor2}
Let $g$ be a radial cscK metric on a complex manifold $M$. Assume that $(M, g)$ can be \K\ immersed into an infinite dimensional complex space forms $(S^\infty, g^\infty_c)$. 
Then 
 
\noindent
(1) If $c\leq 0$ then  $(M, g)$ is a complex space form of non positive holomorphic sectional curvature.

\noindent
(2) If $c>0$ and $g$ is well-behaved  then either  $(M, g)$ is a complex space form of non positive holomorphic sectional curvature or
$M$ is   an open subset of 
$\C P^n$ and    $g=mg_c^n$, $m\in \N^+$.
\end{theor}

\begin{rmk}\rm\label{remarkx}
In (1)  of  Theorem \ref{mainteor2}  we cannot get  to the conclusion that 
the immersion is totally geodesic as in (1)  of Theorem \ref{mainteor}.
Indeed, beside the natural totally geodesic embeddings  $(\C^n, g_0)\rightarrow (\ell^2(\C), g^\infty_0)$
 and $(\CH^n, g_c^n)\rightarrow (\CH^\infty, g_c^{\infty})$ ($c<0$) there exist  \K\ embeddings of  $(\CH^n, g^n_{c})$ into $(\ell^2(\C), g_0)$, for all $c<0$.
 
 Similar considerations hold true for (2) in Theorem \ref{mainteor2}:
 $(\CH^n, g^n_{c'})$ and  $(\C^n, g_0)$ can be \K\ embedded
 into   $(\CP^\infty, g^\infty_{c})$ for all $c'<0$ and $c>0$.
 (The reader is referred to \cite{calabi} for details).
\end{rmk}

Finally we ask what happens when cscK is strengthened to the KE condition.
In this regards we believe the validity of the following:

\vskip 0.3cm
\noindent
{\bf Conjecture 2:}
{\em  A (not well-behaved\footnote{Otherwise one can conclude by (2) of Theorem \ref{mainteor2}.})  radial KE metric 
induced by   $(\C P^\infty, g_c^{\infty})$ (for some $c>0$) is a complex space form.} 

\vskip 0.3cm

In order for the conjecture to make sense we exhibit  in Example \ref{exKENWB}  of Subsection \ref{Examples}  a radial KE metric which is not 
well-behaved.

Notice   that Conjecture 2  is  false without the radiality assumption. Take, for example, 
 any   open contractible open subset  $U$ of a flag manifold $(F, h)$ different from the complex  projective space 
 such that $h$ is projectively induced (see e.g. \cite{TA78})
 then $(U, h_{|U})$ is a KE manifold which admits a global \K\ potential and $h_{|U}$ is  projective induced. 
Another example is given by  a bounded symmetric domain equipped with its Bergman metric
 or any bounded homogeneous domain with a suitable multiple of a   homogeneous metric (see \cite{AMOSSA}). 
 For an example of  non radial KE metric  which 
 \lq\lq does not come'' from an homogeneous one the reader is  referred to \cite{LZ}  where one can find a continuous family  of complete and nonhomogeneous  KE submanifolds of the infinite dimensional complex projective space (see also \cite{VARI} for further examples).

Notice also  that Conjecture 2 turns out to be true for  Ricci flat  metrics on complex surfaces (as explained in Example \ref{EH} of Subsections \ref{Examples} below).

 In Theorem \ref{mainteor3} we show the validity  of  Conjecture 2  under a natural  stability
assumption which the authors of the present paper have already considered in \cite{LSZ}.

\begin{defin}\label{defst}
Let $c>0$.
 A \K\  metric $g$  is said to be  {\em $c$-stable projectively induced} if there exists $\epsilon >0$ such that $\alpha g$ is induced  by $(\C P^\infty, g_c^\infty)$
for all $\alpha \in (1-\epsilon, 1+\epsilon)$.   A \K\  metric $g$  is said to be unstable if it is not $c$-stable projectively 
induced for any $c>0$. When $c=1$ we simply say that $g$ is {\em stable-projectively induced}.
\end{defin}

The reader is referred to \cite{LSZ} 
for details, examples and further properties of stable projectively induced metrics.
Notice that the Fubini-Study metric $g_{FS}$, and more generally any projectively induced metric on a compact manifold
  is unstable, 
while the flat metric $g_0$ and the hyperbolic metric $g_{hyp}$ 
are $c$-stable projectively induced for all $c>0$ 
due to the last part of Remark  \ref{remarkx}.

\begin{theor}\label{mainteor3}
Let $g$ be a radial KE metric
induced by  $(\C P^\infty, g_c^{\infty})$ for some $c>0$.
If $g$ is $c$-stable projectively induced then $(M, g)$ is a complex space form of non-positive  holomorphic sectional curvature.
\end{theor}

We point out that Theorem \ref{mainteor3}  extends to arbitrary radial KE metrics  \cite[Theorem 1.1]{LSZ} valid in the  Ricci flat case.

\vskip 0.3cm
The paper is organized as follows. In Section \ref{secradext}
we summarize some basic facts on extremal \K\ radial metrics. 
In  particular we recall in Lemma \ref{lemmasimple} that these  metrics on an $n$ dimensional  complex manifold
can be described by a rational  family $\psi (y)$, 
$y(r)=rf'(r)$, $r\in (r_{\inf}, r_{\sup})$, depending on four real  parameters 
$A, B, C, D$  (in particular the vanishing of $A, B$ and $D$ is equivalent
to the constancy of the holomorphic sectional curvature of the metric involved).
In Subsection \ref{Examples} we provide many examples (some already mentioned above) of radial extremal metrics by fixing some values of the parameters
and finding their  explicit \K\ potentials.
Finally, the last section  (Section \ref{secproofs}) is dedicated to the proofs of our main results. In Subsection \ref{subeps}, after   briefly recalling  the concept of $\epsilon$-resolvability ($\epsilon=-1, 0, +1$) 
of rank $N\leq\infty$ of a real analytic \K\ metric $g$ and Calabi's
criterium for the inducibility of  $g$ into a finite or infinite dimensional complex space form of hyperbolic, flat or elliptic type (depending on the sign of $\epsilon$), we  specialize  to the case of radial \K\ metrics (Lemma \ref{resol}).
In Subsection \ref{subQ} to a given   radial extremal metric $g$ we associate
a sequence of  rational functions $Q_k^\epsilon (y)$, $k\geq 1$, which are the key tools in the proof of our theorems. 
This is reason why in Lemma \ref{lemlem}
we deeply analyze these functions and their  higher and lower degree coefficients in terms of $A, B, C$ and $D$ of the extremal metric involved.
Finally, in Subsection \ref{subproofs} one can find the proofs of Theorem \ref{mainteor}, Theorem \ref{mainteor2}
and Theorem \ref{mainteor3}.

\vskip 0.3cm
We would like to thank Miguel Abreu for his interest in our work 
and for stimulating discussions about extremal radial metrics.

\section{Radial extremal metrics}\label{secradext}
Let $g$ be a radial  \K\ metric on a complex manifold $M$, equipped with complex coordinates $z_1, \dots ,z_n$.
Let  $\omega =\frac{i}{2}\partial\bar\partial f(r)$ denotes its associated \K\ form
 where $f:(r_{\inf}, r_{\sup})\rightarrow \R$, $r = |z_1|^2 + \cdots + |z_n|^2$, $0\leq r_{\inf}<r<r_{sup}$
 and $(r_{\inf}, r_{\sup})$ is the maximal domain where the radial potential $f$ is defined.

 It is not hard to see that 
the matrix of the metric $g$ and its inverse read  as
\begin{equation}\label{metric}
g_{i\bar j}=f''(r) \bar z_i z_j+f'(r)\delta_{ij}, \ \ \ g^{i\bar j}=\frac{\delta_{ij}}{f'(r)} -\frac{f''(r)}{f'(r) (rf'(r))'}\bar z_j z_i.
\end{equation}
Set 
\begin{equation}\label{y(r)}
y(r):=rf'(r).
\end{equation}
and
\begin{equation}\label{psi(y)}
\psi(r) := ry'(r).
\end{equation}
The fact that $g$ is a metric is equivalent to 
$y(r)>0$ and $\psi (r)>0$, $\forall r\in (r_{\inf}, r_{\sup})$.
Then 
\begin{equation}\label{limtr}
\lim_{r\rightarrow r_{\inf^+}}y(r)= y_{\inf}
\end{equation}
is a non negative real number.
Similarly set
\begin{equation}\label{limtrbis}
\lim_{r\rightarrow r_{\sup^-}}y(r)= y_{\sup}\in (0, +\infty].
\end{equation}

Therefore we can invert the map
$$(r_{\inf}, r_{\sup})\rightarrow  (y_{\inf}, y_{\sup}),\  r\mapsto y(r)=rf'(r)$$ 
on $(r_{\inf}, r_{\sup})$ and think $r$ as a function of $y$, i.e. $r=r(y)$.

The following lemma provides a classification of radial extremal \K\ metrics.
Even if its proof  is   known (see, for example, \cite{abreuC1} and also \cite{XD}) 
we include  it here  for reader's convenience.

\begin{lem}\label{lemmasimple}
Set 
$$\psi (y):=\psi (r(y)).$$
A radial \K\ metric  $g$ is extremal if and only if
\begin{equation}\label{ypsi}
\psi(y) = y - \frac{A}{y^{n-1}} - \frac{B}{y^{n-2}} - C y^2 - D y^3. 
\end{equation}
for some $A,B,C,D \in \R$.
Moreover, the following facts hold true:
\begin{itemize}
\item [(a)]
$g$ is a cscK metric\footnote{with constant scalar curvature equal to  $Cn(n+1)$.} iff $D=0$ and the sign of 
the scalar curvature is equal to the sign of $C$;
\item [(b)]
$g$ is a KE metric with Einstein constant $\lambda$ iff  $B=D=0$  and $C = \frac{\lambda}{2(n+1)}$;
\item [(c)]
$g$ has constant holomorphic sectional curvature 
iff $A=B=D=0$.
\end{itemize}
\end{lem}
\begin{proof}
From (\ref{metric}), we easily get 
$$\det \left(g_{i\bar j}\right)(r)=\frac{(y(r))^{n-1}\psi (y)}{r^{n}}.$$
By  straightforward computations (see e.g. \cite{LSZ} for details) 
we can  compute the Ricci tensor's components
\begin{equation}\label{ricci}
\Ric_{i\bar j}(r)=-\frac{\de^2 \log\det \left(g_{i\bar j}\right)}{\de z_i\de\bar z_j}=\frac{-\frac{d\sigma}{dy}\psi(y)+\sigma(y)-n }{r^2}\bar z_i z_j+\frac{n-\sigma(y) }{r}\delta_{ij},
\end{equation}
and the scalar curvature $s$ of $g$ as a function of $y$
\begin{equation}\label{scalar}
s(y)=\sum_{i, \bar j=1}^n g^{i\bar j}Ric_{i \bar j}=\frac{n(n-1)}{y}-y^{1-n}\frac{d^2 \left[y^{n-1}\psi(y)\right]}{dy^2},
\end{equation}
where
$$\sigma(y):=y^{1-n}\frac{d \left[y^{n-1}\psi(y)\right]}{dy}=(n-1)\frac{\psi(y)}{y}+\frac{d\psi}{dy}.$$

Now, by definition a \ \K\ metric is extremal if and only if the gradient field
\begin{equation}\label{gradfield}
X = \sum_{i, j = 1}^n g^{i\bar j}  \frac{\partial s}{\partial \bar z_j} \frac{\partial}{\partial z_i}
\end{equation}

is holomorphic.

Since the scalar curvature is a radial function, we have $\frac{\partial s}{\partial \bar z_j} = s'(r) z_j$: from this and (\ref{metric}) we can rewrite (\ref{gradfield}) as 

\begin{equation}\label{gradfield2}
X = \sum_{i, j = 1}^n \left(\frac{\delta_{ij}}{f'(r)} -\frac{f''(r)}{f'(r) (rf')'}\bar z_j z_i \right) s'(r) z_j \frac{\partial}{\partial z_i} =  \sum_{i = 1}^n \frac{s'(r)}{(rf'(r))'} z_i \frac{\partial}{\partial z_i} 
\end{equation}

It immediately follows that $X$ is holomorphic if and only if $\frac{s'(r)}{(rf'(r))'}  = \frac{s'(r)}{y'(r)} = \gamma_1$ for some constant $\gamma_1 \in \R$, i.e. 

\begin{equation}\label{condextr}
s = \gamma_1 y + \gamma_2
\end{equation}
where $\gamma_2 \in \R$. By (\ref{scalar}), this means 

\begin{equation*}\label{condextr2}
\frac{n(n-1)}{y}-y^{1-n}\frac{d^2 \left[y^{n-1}\psi(y)\right]}{dy^2}= \gamma_1 y + \gamma_2
\end{equation*}

which integrated gives

\begin{equation*}\label{condextr3}
\psi(y) = y -  \frac{\gamma_1}{(n+1)(n+2)} y^3- \frac{\gamma_2}{n(n+1)} y^2 +  \frac{\gamma_3}{y^{n-2}} + \frac{\gamma_4}{y^{n-1}}
\end{equation*}

which is exactly (\ref{ypsi}) for

\begin{equation}\label{ABCD}
A = - \gamma_4, \ \ B = - \gamma_3, \ \ C = \frac{\gamma_2}{n(n+1)}, \ \ D= \frac{\gamma_1}{(n+1)(n+2)}.
\end{equation}

The proof of  (a) follows by \eqref{condextr} and \eqref{ABCD} and that of (b) can be easily obtained by using  \eqref{metric} and \eqref{ricci}.
If  $g$ has constant holomorphic sectional curvature  then in particular is KE ($B=D=0$)
and the constancy of the norm of the Riemannian tensor $|R|^2$ of $g$ implies $A=0$
as it follows for example by using the expression of $|R|^2$ in \cite{LSZbis}.
Finally, if  $A=B=D=0$ then \eqref{y(r)},
\eqref{psi(y)} and  \eqref{ypsi}  yield
$$f''(r)+C(f'(r))^2=0$$
which integrates explicitly and gives a metric with constant holomorphic sectional curvature.
\end{proof}
Let $g$ be a radial extremal metric as above. By setting
$e^t=r$ we deduce by \eqref{y(r)} and \eqref{psi(y)} that 
the function 
$$y(t):=y(e^t)$$
satisfies the ODE 
$$\frac{dy}{dt}=\psi (y(t)),$$
where $\psi (y)$ is given by \eqref{ypsi}.

The following simple  lemma will be  crucial 
in the proof of Proposition \ref{lemAB01} and in Theorem \ref{mainteor2}.
\begin{lem}\label{ANAL}
The following hold true:
\begin{itemize}
\item [(i)]
If   $\lim_{y\rightarrow y_{\inf}^+}\psi (y)\neq 0$
then $y_{\inf} =0$.
\item [(ii)]
If   
$\lim_{y\rightarrow y_{\sup}^-}\psi (y)\neq 0$
then $y_{\sup} =+\infty$.
\end{itemize}
\end{lem}
\begin{proof}
In order to prove (i) assume by contradiction that  $y_{\inf}\neq 0$ in \eqref{limtr}.
Note first that $t_{\inf}:=\lim_{r\to r_{\inf}} \log r=-\infty$:
otherwise  (if  $t_{\inf}\in\R$) the function $y(t)$ could be prolonged to an open interval containing $t_{\inf}$
being 
 the solution of   the   Cauchy problem
\begin{equation}
\begin{cases}
 y'(t)=\psi (y(t)) \\
 y(t_{\inf})=y_{\inf}>0.
\end{cases}
\end{equation}
Thus, by  the continuity of $\psi (y)$ at $y_{\inf}\neq 0$,
$$\lim_{y\rightarrow y_{\inf}^+}\psi (y)=\lim_{t\rightarrow -\infty}\psi (y(t))=\lim_{t\rightarrow -\infty} y'(t)=0,$$
where the last equality follows by  \eqref{limtr} when $t_{\inf}=-\infty$,
the desired contradiction.
The proof of  (ii) is obtained similarly by considering $t_{\sup}=\lim_{r\to r_{\sup}} \log r$.
\end{proof}

\begin{rmk}\label{WBbis}\rm
In view of the definition of well-behaveness  (Definition \ref{WB})  (i) of Lemma \ref{ANAL} can be equivalently expressed by saying that if the
a radial metric $g$ is not well-behaved (i.e. $y_{\inf}\neq 0$) then $\psi (y_{\inf})=0$.
\end{rmk}

\subsection{Examples}\label{Examples}
\begin{example}\label{exfond}\rm
Consider the extremal  radial metric obtained by taking
$A=B=0$,  
$C=-3$ and $D=-2$ in \eqref{ypsi}.
In this case we can solve explicitly the ODE equation
 \begin{equation*}\label{explInt}
\psi(y) = \frac{d y}{d t} =  y + 3 y^2 + 2 y^3 = y(y+1)(2y+1)
\end{equation*}
and setting $r=e^t$,  we find a unique solution (up to change of complex coordinates) given by the   \K\ potential
$$f(r)  =  \log\left[ \frac{1 - \sqrt{1-4r}}{2r}\right] ,\ 0< r<\frac{1}{4},$$
which shows that $g$ well-behaved being  defined at  $r=0$.

Fix $n\geq 1$ and consider the open unit disk of $\C^n$
of radius $\frac{1}{2}$, namely 
$$M=\{z=(z_1, \dots, z_n)\in \C^n \ | \ 0< |z|^2< \frac{1}{4}\}$$
equipped with the \K\ metric $g$ whose associated  \K\ form
is $\omega =\frac{i}{2}\partial\bar\partial f(r)$.

In order to construct a \K\ immersion
of $(M, g)$
 into $(\C H^{\infty}, g_{hyp})$
consider the function
$$1-e^{-f}=\frac{1}{2}\left( 1- \sqrt{1-4r}\right).$$
Now, recall the Taylor expansion 
$$\sqrt{1+x} = \sum_{k=0}^{\infty} { 1/2 \choose k} x^k,$$
where 
$${ 1/2 \choose k} = \frac{1/2 (1/2  - 1) \cdots (1/2  - k + 1)}{k!}.$$

Therefore

$$1-e^{-f}=\frac{1}{2}\left(1-1-\sum_{k=1}^\infty (- 4)^k{ 1/2 \choose k}r^k\right)= \frac{1}{2}\sum_{{k=1}}^\infty  4^k\left|{ 1/2 \choose k}\right|r^k.$$

By replacing $r = |z_1|^2 + \cdots + |z_n|^2$ then one finds an explicit \K\ immersion via monomials 
 into $(\CH^{\infty}, g_{hyp})$:

$$z=(z_1, \dots, z_n)\mapsto\left(\mathellipsis,\sqrt{\frac{4^{|j|}}{2}\left|{ 1/2 \choose |j|}\right|\frac{|j|!}{j!}}z^j,\mathellipsis\right)_{j\in\N^n,|j|\geq1},$$

where,  for $j=(j_1, \dots j_n)\in \N^n$ we set $z^j:=z_1^{j_1}\dots z_n^{j_n}$, $j!:=j_1!\dots j_n!$, $|j|:=j_1+\cdots +j_n$.

By multiplying the metric $g$ by a positive constant one then obtain \K\ immersion of $(M, g)$ into   $(\C H^{\infty}, g_{c'}^\infty)$ for all $c'<0$ and hence  
into $(\ell^2(\C), g_0)$ and $(\C P^\infty, g_{c}^\infty )$, for all $c>0$  (cfr. \cite[Lemma 5 and Lemma 8]{DHL}).

It remains an open and interesting problem to classify all  extremal radial metrics induced by infinite dimensional complex space form.

\end{example}

\begin{example}\rm\label{simanca}
By taking $n=2$,   $A=C=D=0$, $B=1$ in \eqref{ypsi} one gets
\begin{equation*}\label{explInt}
\psi(y) = \frac{d y}{d t} =  y-1
\end{equation*}
which can be easily integrated to find the 
 \K\ potential
\begin{equation*}
f_{BS}(r)=r+\log r, \  0< r<+\infty.
\end{equation*}

The scalar (not Ricci) flat    \K\ metric  $g_{BS}$ corresponding to this potential is the celebrated  {\em Burns-Simanca metric}.
Notice that $g_{BS}$ is not well-behaved since $rf'(r)=r+ 1\rightarrow 1$ for $r\rightarrow 0^+$.
One can show that $g_{BS}$ is  projectively induced: an explicit \K\ immersion can be found in \cite{LSZ} (see also \cite{LSZbis} and \cite{LMZ}). 
Moreover, one can easily  check that 
$g_{BS}$ cannot be induced by any complex space form of non positive holomorphic sectional curvature in accordance
with   (1) of Theorem \ref{mainteor2}.
\end{example}

\begin{example}\rm\label{partbal}
It is not hard to see that the  radial \K\ metric 
corresponding to the \K\ potential
$$f(r)=\log r-\log (1 -  r^3)$$
provides 
an example of not well-behaved  infinitely projectively induced radial cscK (not KE)  metric
on the punctured disk of $\C^2$
with {\em negative} scalar curvature ($s=-24$).
\end{example}

\begin{rmk}\rm
By the previous two examples it is natural to see
if there  exist projectively induced 
not well-behaved cscK  radial \K\ metrics 
with {\em positive} scalar curvature.
At the moment we do not know any example of such metrics.
\end{rmk}

\begin{example}\rm\label{EH}
In  order to describe all  the radial  Ricci flat  metrics 
one has to solve the ODE ($B=C=D=0$ in \eqref{ypsi}):
\begin{equation*}\label{ricciflat}
\frac{d y}{d t}= \psi(y)  =  \frac{y^n-A}{y^{n-1}}.
\end{equation*}
For either $n=1$ or $A=0$ we get the flat metric so we  assume $n\geq 2$ and $A\neq 0$.
The general  solution of \eqref{ricciflat} is given  by
$y(t)=\left( \gamma e^{nt} +A \right)^{\frac{1}{n}}$ for some $\gamma >0$.
By setting $r=e^t$  we then  get 
$$y(r)=rf'(r)=\left( \gamma r^n +A \right)^{\frac{1}{n}}.$$
By the change of complex variables 
$$z_j\mapsto w_j:=(\frac{\gamma}{|A|})^{\frac{1}{2n}}z_j, j=1, \dots ,n,$$
and still denoting by $r=|w_1|^2+\cdots +|w_n|^2$
we  deduce that   the radial potential $f(r)$ of a Ricci flat  (not flat) metric 
on a  $n$-dimensional complex manifold $n\geq 2$ is, up to the multiplication of a positive constant
($|A|^{-\frac{1}{n}}$),
given by:
\begin{equation}\label{frflat}
f(r)=\int\left(1+ \frac{A}{|A|} r^{-n}\right)^{\frac{1}{n}}.
\end{equation}

If  $A<0$ since $f'(r)>0$ one gets $r_{\inf}=1$ 
and thus $y(r)=rf'(r)\rightarrow 0$ for $r\rightarrow 0^+$ and 
the corresponding radial Ricci flat metric  is well-behaved and  not infinitely projectively induced by (2) of Theorem \ref{mainteor2}.

If  $A>0$ then $r_{\inf}=0$ and $y(r)=rf'(r)\rightarrow 1$
and so $g$ is not  well-behaved.
If we further assume that  $n=2$ one can easily  integrate \eqref{frflat}
and get
$$f(r) = \sqrt{r^2+1} + \ln r - \ln(1 + \sqrt{r^2+1})$$
which is the potential of the celebrated Eguchi-Hanson metric $g_{EH}$
on $\C^2\setminus\{0\}$.
It is not hard to see that $\alpha g_{EH}$ is not projectively induced for all $\alpha\in \R^+\setminus \Z$
(cfr. the proof of \cite[Theorem 1.1]{LSZ}). On the other hand
the first and third authors together with M. Zedda have shown in \cite[Corollary 1]{LOIZEZU}
that  $\alpha g_{EH}$ is not infinitely  projectively induced for all $\alpha \in \Z^+$.
By combining these two facts we deduce  that 
$g_{EH}$ cannot be induced by $(\C P^\infty, g_c^\infty)$ for all $c>0$.
This shows the validity of Conjecture 2 when $n=2$.
The case $n>2$ and $A>0$ still  remains open.
\end{example}

\begin{example}\rm\label{exKENWB}
Let $F:(1, +\infty)\rightarrow \real$ be given by
$$F(y)=e^{-\frac{2}{y+2}}\left[\frac{y-1}{y+2}\right]^{\frac{1}{3}} \ 1<y<+\infty.$$
Define 
$$y(r)=F^{-1}(r), \ 0<r<1.$$
One can easily verify that $y(t)$ with $t=\log r$
satisfies the ODE equation 
$$\frac{dy}{dt}=\psi (y)=y-\frac{4}{3y}-\frac{1}{3}y^2$$
i.e. we take  $n=2$, 
$A=\frac{4}{3}$
$C=-\frac{1}{3}$, $B=D=0$ in \eqref{ypsi}.
By Lemma \ref{lemmasimple}  we then get a radial KE metric $g$ with negative Einstein constant ($\lambda =-2$) on a two dimensional complex manifold.
Moreover  $g$ is not  well-behaved since one can easily check that   $y\rightarrow 1$  
as $r\rightarrow r_{\inf}^+=0$. One can prove (cfr. Remark \ref{Q11} below) that $g$ is not projectively induced in accordance with Conjecture 2.

\end{example}

\begin{example}\rm\label{exrinf}
In this last example we construct a KE radial metric with $r_{\inf}\neq 0$ (and hence $t_{\inf}\neq -\infty$).
Let $F:(0, 1)\rightarrow \R$ be given by
$$F(y)=\log \left[\frac {\sqrt{2y^2+y+1}}{(1-y)^{\frac{1}{4}}}\right], \ 0<y<1$$
One can check that  
$$y(t)=F^{-1}(y(t)), \ -\frac{3}{\sqrt{7}}\arctan\left(\frac{1}{\sqrt{7}}\right)<t<+\infty,$$
satisfies the ODE 
$$\frac{dy}{dt}=\psi (y)=y-\frac{1}{y}-2y^2,$$
namely we take  $n=2$, $A=-1$, $C=2$, $B=D=0$ in \eqref{ypsi}.

By Lemma \ref{lemmasimple}  we then get a radial KE metric $g$ with positive Einstein constant ($\lambda =12$) on a two dimensional complex manifold.
Moreover 
$$r_{\inf}=e^{-\frac{3}{\sqrt{7}}\arctan\left(\frac{1}{\sqrt{7}}\right)}\neq 0$$
and  $g$ is well-behaved since one can easily check that   $y\rightarrow 0$  
as $r\rightarrow r_{\inf}^+$.

Finally, notice that  $g$ cannot be induced by any finite or infinite dimensional  complex space form as it follows by  Theorem \ref{mainteor} and Theorem \ref{mainteor2}.
\end{example}

\section{The proofs of the main results}\label{secproofs}

\subsection{Radial metrics induced by  complex space forms}\label{subeps}
Let $\epsilon\in\{-1,0,1\}$. Following Calabi \cite{calabi}  we say  that a \K\ metric $g$ on a complex manifold 
$M$ is $\epsilon$--resolvable of  rank $N\in\N\cup\{\infty\}$ at  $p\in M$ if  the  matrix $(B_{jk})$ defined by considering the expansion around the point $p$ of 
\begin{equation}\label{gendiast}
\epsilon(e^{\epsilon D_p(z) } -1)+(1-\epsilon^2) D_p(z)=\sum_{m_j ,m_k\in\N^n} B_{ jk} (z - p)^{m_j} (\bar z - \bar p)^{m_k},\end{equation}
 is positive semidefinite and its rank is N, where  $D_p(z)$ is Calabi's diastasis function (cfr. \cite{calabi}, \cite{LoiZedda-book}). Here, $z^{m_j}$ denotes the monomial in $n$ variables $\prod_{\alpha=1}^n z_\alpha^{m_ {\alpha, j}}$ and we arrange every $n$-tuple of nonnegative integers as a sequence $m_j = (m_{1, j} , \mathellipsis ,m_{n, j} )$ such that $m_0 = (0, \mathellipsis , 0)$, $|m_j|\leq |m_{j+1}|$ for all positive integer $j$ and all the $m_j$'s with the same $|m_j |$ using lexicographic order. 
Moreover Calabi's criterium affirms that   a \K\ metric $g$ is $\epsilon$--resolvable of  rank $N\in\N\cup\{\infty\}$ at  $p$ if and only if there exists a neighborhood of $p$ that can be  holomorphically and isometrically (\K) immersed respectively in $(\CH^N,g_{hyp})$ for $\epsilon=-1$, $(\C^N,g_0)$ for $\epsilon=0$ and $(\CP^N,g_{FS})$ for $\epsilon=1$. 
Notice that if  $M$ is connected, then the property of resolvabilty does not depend on the choice of the point $p$.

When the metric $g$ is radial with associate \K\ form $\omega =\frac{i}{2}\partial\bar\partial f(r)$, $r\in(r_{\inf}, r_{\sup})$
(as in the previous section) by using Calabi's criterium (see \cite{calabi}) one can prove
the following result which can be obtained  by following the same outline of  \cite[Lemma 2.2]{LSZ} where 
the authors of the present paper consider the  \K\ immersions  of radial \K\ metrics into $(\C P^\infty, g_{FS})$ (namely the $1$-resolvability of infinite rank).

\begin{lem} \label{resol}
Let  $g$ be a radial extremal metric  on a complex manifold $M$ of complex dimension $n$.
Set
\begin{equation}\label{Feps}
F_\epsilon (r)=\epsilon e^{\epsilon f(r) } +(1-\epsilon^2) f(r), \ r\in (r_{\inf}, r_{\sup}).
\end{equation}
If $g$ is   $\epsilon$-resolvable then the following facts hold true:
\begin{itemize}
\item
If $n=1$,
$$\det\left(\frac{\partial^{\alpha+\beta}\F}{\partial z^\alpha\partial\bar z^\beta} \right)_{1\leq\alpha,\beta\leq I}\geq0,\qquad\forall I\in\Z^+.$$
\item
If $n\geq 2$,
$\frac{d^k F_\epsilon}{dr^k}\geq 0$,  
for every positive  index $k$. 
\end{itemize}
Moreover  if $g$ is $\epsilon$-resolvable of finite rank,  there exists an index $I$ such that if $n=1$ one has
$\det\left(\frac{\partial^{\alpha+\beta}\F}{\partial z^\alpha\partial\bar z^\beta} \right)_{1\leq\alpha,\beta\leq h}\equiv 0$
and if $n\geq 2$ one has
$\frac{d^h F_\epsilon}{dr^h}\equiv 0$, $\forall h>I$.   
 \end{lem}

\subsection{The rational  functions $Q_k^\epsilon (y)$}\label{subQ}
Given a  radial metric $g$ as above,  by \eqref{Feps} it  is easy to prove by induction that, for $\epsilon=\pm 1$,
\begin{equation}\label{per1}
\frac{d^k \F}{dr^k} = \epsilon g^\epsilon_k(r)\F(r),
\end{equation}
where $g^\epsilon_k(r)$ is a function of the derivatives of $f(r)$ determined by the following recursive definition
\begin{equation}\label{geps1}
g_1^\epsilon (r) = f'(r);\qquad g_{k+1}^\epsilon(r) = \left(g_k^\epsilon\right)' (r)+ \epsilon f'(r)g_k^\epsilon(r).
\end{equation}
Moreover,  we have
\begin{equation}\label{per2}
\frac{d^k F_0}{dr^k} =\frac{d^k f}{dr^k}=g^0_k(r),
\end{equation}
and \eqref{geps1} holds true for $\epsilon=0$, i.e.
\begin{equation*}
g_1^0(r)= f'(r) ;\qquad g_{k+1}^0(r) = \left(g_k^0\right)'.
\end{equation*}

By  setting as before
$r = e^t$,  $y(r)=rf'(r)$ and $\psi(y)=r(rf')'$,
we can rewrite the recursive formula \eqref{geps1}  as 
\begin{equation}\label{gkdef}
g_k^\epsilon (r) = \frac{Q_k^\epsilon(y)}{r^k}
\end{equation} 
where
\begin{equation}\label{Qk}
Q_1^\epsilon  (y):= y;\qquad Q_{k+1}^\epsilon (y) = (\epsilon y - k)Q_k^\epsilon (y) + \frac{dQ_k^\epsilon}{dy}\psi (y).
\end{equation}

\begin{rmk}\rm\label{rmkQ2FS}
If $g$ is extremal and $Q_2^\epsilon (y)$ vanishes identically on $(y_{\inf}, y_{\sup})$ then $(M, g)$ is a complex space form.
Indeed by  \eqref{Qk} with $k=1$ we get $\psi (y)=y-\epsilon y^2$
and so by Lemma \ref{lemmasimple}, $g$ has constant holomorphic sectional curvature.
\end{rmk}

Notice that  for each $k$, $Q_k^\epsilon (y)$ are functions defined in the open interval $(y_{\inf}, y_{\sup})$
We claim that
\begin{equation}\label{Pk}
 Q_{k}^\epsilon (y) =y\prod_{j=1}^{k-1}(\epsilon y- j) + \frac{\psi  (y)P_k^\epsilon (y)}{y^{(k-2)n}},
\end{equation}
where $P_k^\epsilon(y)$ is a polynomial in 
$y$ with coefficients depending on $A, B, C, D$ and 
$\epsilon$ and the above product  equals $1$ when $k=1$.

We prove our claim by induction. 
For $k= 1$, we have $Q_1^\epsilon(y)  = y$ so that \eqref{Pk}  is verified with $P_1^\epsilon (y) = 0$. Now, assume that \eqref{Pk} holds true for some $k>1$. Then, by using \eqref{Qk} 
 one can easily verify that 
\begin{equation*}
Q_{k+1}^\epsilon(y) =  y\prod_{j=1}^{k}(\epsilon y- j) + \frac{\psi(y)  P_{k+1}^\epsilon (y)}{y^{(k-1)n}},
\end{equation*}
where
\begin{equation}\label{recursive}P_{k+1}^\epsilon(y)=y^n(\epsilon y-k)P_k^\epsilon (y) + y^{(k-1)n} \frac{d}{dy} \left( y \prod_{j=1}^{k-1}(\epsilon y- j) \right)+R^\epsilon_k(y),\end{equation}
and 
$$R_k^\epsilon(y)=y\frac{d\left[y^{n-1}\psi(y)\right]}{dy}P_k^\epsilon (y)+y^{n}\psi(y)\frac{dP_k^\epsilon}{dy} -\left[(k-1)n-1\right]y^{n-1}\psi(y) P_k^\epsilon (y).$$
Observe that $R_k^\epsilon(y)$ is a polynomial in $y$ since  $y^{n-1}\psi(y)$ is a polynomial by  \eqref{ypsi}.
Thus $P_{k+1}^\epsilon (y)$ is a polynomial in $y$ proving our claim.

Notice  that $Q_k(y)$ can be written as a finite expression
\begin{equation}\label{Qknew}
Q_k^\epsilon(y)=\sum_{h=-s}^tq_hy^h, \ s, t\in\N,
\end{equation}
where $q_h:=q_h (k, A, B, C, D, \epsilon)$ are real numbers
such that $q_h=0$ for all $h<-s$
and $h>t$.
We say that $t$  is the  {\em degree} of $Q_k^\epsilon(y)$, $q_t$ its {\em leading term}, $-s$ its {\em lower degree} and $q_s$ its {\em lower term}, respectively.

The properties of $Q_k^\epsilon(y)$, $y\in (y_{\inf}, y_{\sup})$, needed in the proof of our main results are summarized in the following two lemmata (Lemma \ref{leml} and Lemma \ref{lemlem}) and in the two corresponding propositions  (Proposition \ref{lemAB0} and Proposition \ref{lemAB01}).
\begin{lem}\label{leml}
Let  $g$ be a radial extremal metric  on a complex manifold $M$ of complex dimension $n$. Assume that $g$ is $\epsilon$-resolvable, then  we have:
\begin{itemize}
\item [(a1)]
If $n=1$,
$$\epsilon^I\det\left(\sum_{i=0}^\alpha {\alpha\choose i}\frac{\beta!}{(\beta-i)!}Q_{\alpha+\beta-i}^\epsilon \right)_{1\leq\alpha,\beta\leq I}\geq0,\quad\forall I\in\Z^+ \text{if $\epsilon\neq 0$,}$$
$$\det\left(\sum_{i=0}^\alpha {\alpha\choose i}\frac{\beta!}{(\beta-i)!}Q^0_{\alpha+\beta-i} \right)_{1\leq\alpha,\beta\leq I}\geq0,\quad\forall I\in\Z^+ \text{if $\epsilon= 0$.}$$
\item [(a2)]
If $n\geq 2$,
$Q_k^\epsilon (y)\geq 0$,
for every positive  index $k$. 
\end{itemize}
Moreover  if $g$ is $\epsilon$-resolvable of finite rank, then
there exists an index $I$ such that:
\begin{itemize}
\item [(b1)]
If $n=1$,
\begin{equation}
\label{dim1}\det\left(\sum_{i=0}^\alpha {\alpha\choose i}\frac{\beta!}{(\beta-i)!}Q_{\alpha+\beta-i}^\epsilon \right)_{1\leq\alpha,\beta\leq h}=0, \ \forall h\geq I,
\end{equation}
\item [(b2)]
If $n\geq 2$,
\begin{equation}\label{indexI}
 Q_k^\epsilon (y)\equiv 0, \ \forall h\geq I.
 \end{equation} 
 \end{itemize}
 \end{lem}
 \begin{proof}
 Notice that for $n=1$ and $\beta\geq\alpha$ we can write
$$ \frac{\partial^{\alpha+\beta}\F}{\partial z^\alpha \partial\bar z^\beta}= \frac{\partial^{\alpha}}{\partial z^\alpha }\left(z^\beta \frac{d^{\beta}\F}{d r^\beta } \right)=\sum_{i=0}^\alpha {\alpha\choose i}\frac{\beta!}{(\beta-i)!}z^{\beta-i}\bar z^{\alpha-i}  \frac{d^{\alpha+\beta-i}\F}{d r^{\alpha+\beta-i} }.$$
Thus (a1) and (b1)  follow by taking into account \eqref{per1}, \eqref{per2}, \eqref{gkdef} and    Lemma \ref{resol} for $n=1$.
Similarly (a2) and (b2)  follow by  Lemma \ref{resol} for $n\geq 2$.
\end{proof}

\begin{rmk}\rm\label{Q11}
Using Lemma \ref{leml} one can show that some specific radial \K\ metric cannot be induced by  a complex space form.
For example the KE metric $g$ in Example \ref{exKENWB} is not projectively induced since one can check via computer's aid that the associated rational function $Q^1_{11}(y)$
is stricly negative on a right neighborhood of $y=1$. In order to give further evidence of  the validity of Conjecture 2  one could try to show that $g$ cannot be induced by  $(\C P^\infty, g_c^\infty)$, for all $c>0$, or equivalently to show that $\alpha g$ in not projectively induced for any $\alpha >0$. This does not seem to be  an easy task. 
\end{rmk}

\begin{lem}\label{lemlem}
Let  $g$ be a radial extremal metric  on a complex manifold $M$ of complex dimension $n\geq 2$. Then
for  $k\geq 2$ we have\footnote{For $k=1$, one has $Q_1^\epsilon(y)=y$ and hence its  the leading and the lower term concide and are equal to $1$.}:
\begin{itemize}
\item
 the degree of $Q^\epsilon_k(y)$ is equal to $2k-1$ and its  leading term  is
\begin{equation}\label{coeffD}
-D^{k-1}\prod_{j=2}^{k-1 }(1-2j);
\end{equation}
\item
the   lower degree of $Q^\epsilon_k(y)$  is $n(1-k)+1$ and its lower term  is
\begin{equation}\label{coeffA}
-A^{k-1}(k-2)!\prod_{j=1}^{k-2}\left(n-\frac{1}{j}\right).
\end{equation}
\end{itemize}

In particular 
\begin{equation}\label{k2}
Q^\epsilon_2(y)= - \frac{A}{ y^{ n-1}} - \frac{B}{ y^{ n-2}}
+(\epsilon-C) y^2- D y^3 \end{equation}
and 
$$Q_3^\epsilon (y)=
\frac{A^2(1-n)}{y^{2n-1}}+\frac{AB(3-2n)}{y^{2n-2}}+
\frac{B^2(2-n)}{y^{2n-3}}+\frac{A(n+1)}{y^{n-1}}+$$
$$+\frac{A\left[C\left(3-n\right)-3\epsilon\right]+Bn}{y^{n-2}}+
\frac{(AD+BC)(4-n)-3B\epsilon}{y^{n-3}}+
\frac{BD(5-n)}{y^{n-4}}+$$
\begin{equation}\label{k3}
+(2C^2-3C\epsilon-D+\epsilon^2)y^3+D(5C-3\epsilon)y^4+
3 D^2 y^{5}.
\end{equation}

Moreover, the following  facts hold true:
\begin{itemize}
\item [(i)]
when $D=0$ the degree of $Q_k^\epsilon(y)$ 
is equal to $k$  and its  leading term is 
$$(-1)^{k-1}(k-1)!\prod_{j=1}^{k-1}\left( C-\frac{\epsilon}{j}\right).$$
\item [(ii)]
when $A=0$, the lower degree of  $Q_k^\epsilon(y)$ 
is equal to  $n+k-nk$ and its lower term is 
$$-B^{k-1}(k-2)!\prod_{j=1}^{k-2}\left(n-\frac{j+1}{j}\right).$$
\end{itemize}
\end{lem}
\begin{proof}
 It  can be  obtained
 by  straightforward computations, using
\eqref{Pk} and \eqref{recursive}
and the induction on $k$.
\end{proof}

\begin{prop}\label{lemAB0}
Let $g$ be  a radial extremal metric on a complex manifold $M$
of complex dimension $n\geq 2$.
 Assume that $g$ is well-behaved and 
that  $g$ is projectively induced. Then  then $A=B=0$ in (\ref{ypsi}).
\end{prop}

\begin{proof}
Since $g$ is well-behaved $y_{\inf}=0$.
Assume by a contradiction that $A\neq 0$.
By  equation \eqref{k3} with $\epsilon =1$  one gets 
$$y^{2n-1}Q^1_3(y)\rightarrow A^2(1-n)< 0,\  \  \mbox{for}\ \ y\rightarrow y_{\inf}^+=0^+.$$
Then we deduce that $Q^1_3(y)$ would be negative in a right neighborhood of $0$ in contrast with Lemma \ref{leml}.
Expression \eqref{k3} with $A=0$ yields
$$y^{2n-3}Q^1_3(y)\rightarrow B^2(2-n),\  \  \mbox{for}\ \ y\rightarrow 0^+.$$
Then, if $n>2$, by the same argument just used to show the vanishing of $A$, one sees that it must be $B=0$.

It remains to treat the case $n=2$. On the one hand  $Q_3^1 (y)$ with $A=0$ and $n=2$ becomes a polynomial expression with constant term equals to $2B$ and so 
$$Q^1_3(y)\rightarrow 2B,\  \  \mbox{for}\ \ y\rightarrow 0^+.$$
Then by Lemma \ref{leml} one deduces  that $B \geq 0$.
 On the other hand (\ref{k2}) with $A=0$ and $n=2$ rewrites 
$$Q_2^1(y) =  -B - D y^3 + (1-C)y^2.$$
 Again by letting $y \rightarrow 0^+$ and by Lemma \ref{leml}  one gets $B \leq 0$. So we deduce $B=0$. The proposition  is proved.
\end{proof}

\begin{rmk}\rm
Example \ref{exfond} in Subsection \ref{Examples} with $n=1$  shows that the assumption $n\geq 2$  in Lemma \ref{lemAB0}   is necessary. 
Moreover   Example \ref{simanca} and Example \ref{partbal}  indicate that in  the lemma the well-behaveness condition cannot be dropped.
\end{rmk}

\begin{prop}\label{lemAB01}
Let $g$ be  a radial extremal metric on a complex manifold $M$
of complex dimension $n\geq 2$.
 If $g$ is $\epsilon$-resolvable with $\epsilon\leq 0$ then 
 $A=B=0$ in (\ref{ypsi}).
\end{prop}
\begin{proof}
By the very  definition of $\epsilon$-resolvability the \K\ manifold $(M, g)$ can be \K\ immersed into the finite or infinite dimensional flat or complex hyperbolic space.
It follows either  by Remark \ref{remarkx}  in the finite dimensional  case or by  \cite[Lemma 5 and Lemma 8]{DHL} in the infinite dimensional case that   $g$ is infinitely  projectively induced.
Thus, the proof will be ended if we show that $g$ is well-behaved so to apply Proposition \ref{lemAB0}.
Assume by contradiction this is not the case, i.e.  $y_{\inf}>0$. Then by  (i) of 
 Lemma \ref{ANAL} (cfr. Remark \ref{WBbis}) one has $\lim_{y\to y_{\inf}^+} \psi(y)=0$
 which combined with 
\eqref{Pk} for $k=2$ and 
the fact that, by assumption,  the metric is $\epsilon$-resolvable with 
$\epsilon\leq 0$, give $\lim_{y \to y_{\inf}^+}Q_2^\epsilon (y)=- y_{\inf}(|\epsilon| y_{\inf}+1)<0$, in contrast with Lemma  \ref{resol}.
\end{proof}

\subsection{The proofs of Theorem \ref{mainteor}, 
Theorem \ref{mainteor2} and Theorem \ref{mainteor3}.}\label{subproofs}
\begin{proof}[Proof of Theorem \ref{mainteor}]
By multiplying the metric $g$ by
$\frac{c}{2}$ (if $c\neq 0$) we can assume
 the ambient space is one of the following: $(\C H^n, g_{hyp}^n)$, $(\C^n, g_0^n)$, 
$(\C P^n, g^n_{FS})$ and so the metric $g$ is $\epsilon$-resolvable with $\epsilon =-1, 0, 1$, respectively.
In order to prove (1) and (2)  of Theorem \ref{mainteor} it is enough to  show that 
$g$ has constant holomorphic sectional curvature and then to  appeal to  
Calabi's classification \cite{calabi} of  \K\ submanifolds of  finite dimensional complex space forms.

We consider the cases $n=1$ and $n\geq 2$ separately.

\vskip 0.1cm

\noindent
\underline{{Case $n=1$.}}
We are going to  show that  $D=0$: this will suffice since  by (a) of Lemma \ref{lemmasimple}
this would imply $g$ is cscK and hence, since $n=1$, $g$ has constant holomorphic sectional curvature. 
In order to show the vanishing of $D$ observe that 
by  Lemma \ref{lemlem} (in the notation of Lemma \ref{leml}) one has 
$$\deg\left(\sum_{i=0}^\alpha {\alpha\choose i}\frac{\beta!}{(\beta-i)!}Q_{\alpha+\beta-i}^\epsilon (y)  \right)=\deg Q_{\alpha+\beta}^\epsilon(y)=2(\alpha+\beta)-1.$$
Since the metric $g$ is $\epsilon$-resolvable of finite rank we can pick 
an index $I$ such that \eqref{dim1} holds true, 
namely
\begin{equation}
\label{dim1bis}\det\left(\sum_{i=0}^\alpha {\alpha\choose i}\frac{\beta!}{(\beta-i)!}Q_{\alpha+\beta-i}^\epsilon (y)\right)_{1\leq\alpha,\beta\leq I}=0.
\end{equation}
If $\sigma$ is an arbitrary chosen permutation on $I$ indices
 then the degree of $\prod_{\alpha=1}^I Q_{\alpha+\sigma(\alpha)}^\epsilon (y)$
 does not depend on the permutation $\sigma$:  indeed
 $$\deg\left( \prod_{\alpha=1}^I Q_{\alpha+\sigma(\alpha)}^\epsilon\right)=2\sum_{\alpha=1}^I\big(\alpha+\sigma(\alpha)\big)-I= 2I^2+I.$$
 
Therefore  the leading term of the left hand side of \eqref{dim1bis}
is given by the determinant of the leading terms of the  $Q_{\alpha+\beta}^\epsilon$.
By \eqref{coeffD} of  Lemma \ref{lemlem} this is given by
$$\det\left(-D^{\alpha+\beta-1}\prod_{j=2}^{\alpha+\beta-1}(1-2j) \right)_{1\leq\alpha,\beta\leq I}$$
and, by a  straightforward computation, this is equal to
$$(-1)^I(-2)^{\frac{I}{2}(I-1)} D^{I^2}\prod_{j=2}^{I}(1-2j)^{I-j+1}\prod_{2\leq j<k\leq I+1}(k-j).$$
Hence, by  \eqref{dim1bis}  $D$ is forced to be $0$.

\vskip 0.1cm

\noindent
\underline{{Case $n\geq 2$.}} Let  $I\in \Z^+$ be the minimal\footnote{The minimality  of $I$ will be used only for the cases $n=2$ and $\epsilon=1$ at the end  of the proof.} index such that 
\begin{equation}\label{QIeps}
Q_I^\epsilon \equiv 0,
\end{equation}
whose existence is guaranteed by \eqref{indexI} in Lemma \ref{leml}.
If $I=2$, and hence $Q_2^\epsilon \equiv 0$,  Remark \ref{rmkQ2FS} implies that  $g$ has constant holomorphic sectional curvature and so  the theorem is proved.  
Hence we can assume $I\geq 3$.
We deduce that  the leading term and the lower term  of $Q_I^\epsilon(y)$
must vanish.
By \eqref{coeffD} and \eqref{coeffA} of Lemma  \ref{lemlem} they are given respectively by
$-D^{I-1}\prod_{j=2}^{I-1 }(1-2j)$
and $-A^{I-1}(I-2)!\prod_{j=1}^{I-2}\left(n-\frac{1}{j}\right)$. Hence we deduce  that $A=D=0$.
By  (c) of Lemma \ref{lemmasimple} 
the proof of the theorem will be completed  if we show that  also $B=0$.
Notice that for $\epsilon\leq 0$,
Proposition \ref{lemAB01} implies that $B=0$.
Therefore it remains to show that $B=0$  when  $\epsilon =1$ (and $A=D=0$).

We distiguish two cases: $n\neq 2$ and $n=2$.

Assume $n\neq 2$.  By  (i)  of Lemma \ref{lemlem} the condition  $D=0$ implies that  the  leading term of  $Q_I^1(y)$ 
is given by
\begin{equation}\label{BI}
-B^{I-1}(I-2)!\prod_{j=1}^{I-2}\left(n-\frac{j+1}{j}\right).
\end{equation}
Therefore (since   $n\neq 2$)  it follows  by \eqref{QIeps} and \eqref{BI}   that $B=0$ and we are done\footnote{
This argument works also for $\epsilon\leq 0$ since \eqref{BI} does not depend on $\epsilon$.}.

Assume $n=2$.  In this case equation  \eqref{ypsi} with $A=D=0$
is a polynomial of degree $2$ in $y$, namely
\begin{equation}\label{parabola}
\psi (y)=-Cy^2+y-B.
\end{equation}
Since  $A=0$ by (ii) of  Lemma \ref{lemlem}  the lower term  of  $Q_I^1(y)$ 
is  given  by 
\begin{equation}\label{CI}
(-1)^{I-1}(I-1)!\prod_{j=1}^{I-1}\left( C-\frac{1}{j}\right).
\end{equation}
 By \eqref{QIeps}  (with $\epsilon=1$) and \eqref{CI}   it follows  that
 \begin{equation}\label{Cset} 
C\in\left\{1,\frac{1}{2},\mathellipsis,\frac{1}{I-1}\right\}.
\end{equation}
In particular $C>0$\footnote{In accordance to the fact that we will show $(M, g)$ is an elliptic complex space form.} and 
since  $\psi(y)>0$  we deduce that 
$$0\leq y_1\leq y_{\inf}<y(t)<y_{\sup}\leq y_2<+\infty,\  \forall t\in (t_{\inf}, t_{\sup}),$$ 
where $y_1$ and  $y_2$  are the   two distinct  roots of \eqref{parabola}.
Moreover \eqref{QIeps} and \eqref{Pk} (with $\epsilon=1$) yield
\begin{equation}\label{Nmin1}\psi(y)P_{I}^1(y)+y^{(I-2)n+1}\prod_{j=1}^{I-1}(y-j)\equiv 0
\end{equation}
from which it follows that 
\begin{equation}\label{sjset}
y_j\in\{0, \dots , I-1\}, \ j= 1, 2.
\end{equation}

The proof of the theorem  will be ended if  $y_1=0$ since in this case $B=0$.
 Let us  suppose by contradiction that $y_1 \neq 0 $.
Thus
 \begin{equation}\label{s2min}
 y_2\in\{0, \dots , I-2\},
 \end{equation}
 being  $y_1+y_2=\frac{1}{C}$    at most equal to $I-1$  by \eqref{Cset}.
 Moreover $y_{\inf}\neq 0$  (since $0\leq y_1\leq y_{\inf}$).
 Thus by  (i) (resp. (ii)) of Lemma \ref{ANAL} it follows $\psi (y_{\inf})=0$ (resp. $\psi (y_{\sup})=0$) and hence $y_1=y_{\inf}$ (resp. $y_2=y_{\sup}$).
 
Since $n= 2$ and  $I-1\geq 2$ we can apply  Lemma \ref{leml}
and \eqref{Qk} to obtain 
\begin{equation}\label{QI-1}
Q_{I-1}^1(y)=y \prod_{j=1}^{I-2}(y-j) +\frac{\psi(y)  P^1_{I-1}(y)}{y^{(I-3)n}}\geq 0, \ \forall
y\in (y_{\inf}, y_{\sup})=(y_1, y_2).
\end{equation}  
and
\begin{equation}\label{QI-1bis}
Q^1_I (y) = (y - I+1)Q_{I-1}^1 (y) + \frac{dQ_{I-1}^1}{dy}\psi (y),
\end{equation}

which combined with \eqref{QIeps}  (with $\epsilon=1$),  \eqref{s2min}, $\psi (y_2)=0$ 
and $\psi (y)>0$ give
$Q_{I-1}^1(y_2)=0$ and 
 $\frac{dQ_{I-1}^1}{dy}\geq 0$,  $\forall y\in (y_{\inf}, y_{\sup})$.
 Therefore  $Q_{I-1}^1\equiv 0$, namely \eqref{QIeps} holds true also for $I-1$ in contrast with the assumption of the minimality of $I$.
This  yields the desired contradiction and concludes the proof  of the theorem.
\end{proof}

\begin{proof}[Proof of Theorem \ref{mainteor2}] 
In order to prove (1) and (2)  of Theorem \ref{mainteor2} we will   show that 
$g$ has constant holomorphic sectional curvature and the proof will follow by 
Calabi's classification \cite{calabi} of  \K\ submanifolds  of 
infinite dimensional complex space forms.

If $n=1$ there is nothing to be proved since in this case a cscK metric 
has constant holomorphic sectional curvature.
So assume $n\geq 2$.
In order to prove (1)  (resp. (2)) of Theorem \ref{mainteor2}
we can assume 
as before 
(by multiplying the metric by a suitable constant)
that  $(M, g)$ admits a \K\ immersion either 
into $(\CH^\infty, g_{hyp})$ or $(\ell^2(\C), g_0)$ (resp. $(\C P^\infty, g_{FS})$).
By Proposition \ref{lemAB01} (resp. 
Proposition \ref{lemAB0}, which can be applied since $g$ is assumed to be well-behaved)   we get $A=B=0$.
By combining this with the hypothesis cscK
($D=0$) and 
by (a) and (c) of  Lemma \ref{lemmasimple}  one deduces
$g$ has constant holomorphic sectional curvature and we are done.
\end{proof}

The key ingredient in the proof of  Theorem \ref{mainteor3} is the following
proposition\footnote{Notice that by the above Conjecture 2 and the fact that any complex space form is well-behaved we believe that the set of  metrics satisfying the assumption
Proposition \ref{mainprop3} is empty.}.

\begin{prop}\label{mainprop3}
Let $g$ be a radial KE metric
on a complex manifold $M$. 
Assume that  $g$ is not well-behaved and infinitely  projectively induced.
Then the Einstein constant of $g$ is a  rational number. In particular $g$ is unstable.
\end{prop}
\begin{proof}
Notice first that the condition that $g$ is  not well-behaved implies $n\geq 2$ since a KE metric on a
complex $1$-dimensional manifold has constant holomoprhic sectional curvature and so it is necessarily 
well-behaved.
We are going to show that 
\begin{enumerate}
\item[(i)] $y_{\inf}\in \Z$,
\item[(ii)] $\tilde n := n - \frac{\lambda}{2}y_{\inf} \in \Z$
\end{enumerate}
which clearly implies the rationality of $\lambda$.

By the very definition of well-behaveness we know  that $y_{\inf}$ is a non zero real number.
Thus  by (i) of Lemma \ref{ANAL}  one has  $\psi(y_{\inf})=0$, i.e. $\psi(y) = (y-y_{\inf}) \tilde \psi(y)$ 
for some rational function $\tilde \psi (y)$. By imposing the KE assumption in (\ref{ypsi}) and using (b) of   Lemma \ref{lemmasimple} one gets 
$$\psi(y) = y - \frac{A}{y^{n-1}} - \frac{\lambda}{2(n+1)} y^2,$$
from which  one immediately finds $\tilde \psi(y_{\inf}) = \frac{d\psi}{dy}(y_{\inf}) = \tilde n$. Notice that from this, 
$\psi(y_{\inf})=0$ and $\psi (y)> 0$, for $y >y_{\inf}$, it follows that $\tilde n$ must be  nonnegative.

Now, combining $\psi(y_{\inf})=0$ with (\ref{Pk}) with $\epsilon = +1$, one immediately deduces that
\begin{equation}\label{QkL}
Q^1_k(y_{\inf}) = y_{\inf}(y_{\inf}-1) \cdots (y_{\inf}-k+1)
\end{equation}
and then, if $y_{\inf} \notin \Z$, one has $Q^1_{[y_{\inf}]+2}(y_{\inf}) < 0$, which by (a2) of  Lemma \ref{leml} contradicts the assumption that the metric is projectively induced. This shows (i).

In order to prove (ii), we show that if $\tilde n \notin \Z$ then $Q^1_{y_{\inf}+[\tilde n] + 2}(y)$ is strictly negative in a right neighbourhood of $y=y_{\inf}$, and the conclusion will follow again by contradiction from  (a2) of  Lemma \ref{leml}.

By (\ref{Qk}), for any positive integer $j$ we easily get 
$$Q^1_{j+1}(y_{\inf})+\tilde n \frac{dQ^1_{j+1}}{dy}(y_{\inf}) =
\left( Q^1_{j}(y_{\inf}) + \tilde n \frac{dQ^1_{j}}{dy}(y_{\inf})\right)(y_{\inf}-j+ \tilde n) $$
from which, since $Q^1_1(y_{\inf}) + \tilde n \frac{dQ^1_1}{dy}(y_{\inf}) = y_{\inf} + \tilde n$ (recall that $Q^1_1(y) = y$), we get 
\begin{equation}\label{QkQk'}
Q^1_{k}(y_{\inf}) + \tilde n \frac{dQ^1_{k}}{dy}(y_{\inf}) = (y_{\inf}+\tilde n - k +1 ) \cdots (y_{\inf} + \tilde n -1)(y_{\inf} +\tilde n),
\end{equation}
for any integer $k\geq 2$.

By taking in particular $\hat k = y_{\inf} + [\tilde n ] +2$ and noticing that from (\ref{QkL}) it follows that $Q^1_{y_{\inf}+j}(y_{\inf})=0$ for any $j \geq 1$, one gets
\begin{equation}\label{QkQk'2}
\tilde n \frac{dQ^1_{\hat k}}{dy}(y_{\inf}) = (\tilde n - [\tilde n] -1 )(\tilde n - [\tilde n]) \cdots (y_{\inf} +\tilde n).
\end{equation}
Thus, by the assumption $\tilde n \notin \Z$ (and $\tilde n>0$), one concludes $\frac{dQ^1_{\hat k}}{dy}(y_{\inf}) < 0$, which together with $Q^1_{\hat k}(y_{\inf}) = 0$ immediately implies that $Q^1_{\hat k}(y)$ is strictly negative in a right neighbourhood of $y=y_{\inf}$, the wished contradiction.
The last part of the proposition follows directly by  the definition of stable projectively induced metric.
\end{proof}

An interesting consequence of Proposition \ref{mainprop3} is the following corollary 
which should be  compared to  a result of  D. Hulin \cite{HUfr} (see also \cite[Theorem 1.1]{LM} for an alternative proof) 
on the rationality of the Einstein constant of a  finite projectively induced KE metric.

\begin{cor}\label{corolmainprop3}
Let $g$ be a radial KE metric   of positive  
Einstein constant $\lambda >0$.
If   $g$ is infinitely    projectively induced 
then $\lambda$ is a  rational number.
\end{cor}
\begin{proof}
On the one hand if $g$ is well-behaved then by (2) of Theorem \ref{mainteor2} $(M, g)$ is a complex space form and the assumption $\lambda >0$ implies is the complex projective space, $g=mg_{FS}$ and  $\lambda =\frac{2(n+1)}{m}\in\Q$. On the other hand if $g$ is not well-behaved the rationality of $\lambda$ is guaranteed  by Proposition \ref{mainprop3}.
\end{proof}

\begin{proof}[Proof of Theorem \ref{mainteor3}] 
By multiplying the metric $g$ by $\frac{c}{2}$ (if $c\neq 0$) 
we can assume  $g$ is infinitely  
and stable projectively induced. 
By Proposition \ref{mainprop3} $g$ is forced to be well-behaved.
Thus   by (2) of Theorem \ref{mainteor2}
$(M, g)$ is a complex space form of non-positive holomorphic sectional curvature (since $g_{FS}$
is unstable).
\end{proof}

\end{document}